\documentclass{amsart}
\usepackage{amsmath,amsthm,amsfonts,amssymb,amscd,systeme}
\usepackage{tikz,tikz-cd}
\usepackage{empheq}
\usepackage{lipsum}
\usepackage{lastpage}
\usepackage{thmtools}
\usepackage{enumerate}
\usepackage[shortlabels]{enumitem}
\usepackage{fancyhdr}
\usepackage{mathrsfs}
\usepackage{xcolor}
\usepackage{graphicx}
\usepackage{mathtools}
\usepackage{listings}
\usepackage{bbm}

\usepackage{hyperref}
\usepackage[makeroom]{cancel}
\usepackage[utf8]{inputenc}

\DeclareMathAlphabet\mathbfcal{OMS}{cmsy}{b}{n}
\DeclareMathOperator{\Sl}{SL}
\DeclareMathOperator{\Gl}{GL}

\DeclareMathOperator{\im}{Im}
\DeclareMathOperator{\re}{Re}

\DeclareMathOperator{\Tr}{Tr}

\setlist[enumerate]{leftmargin=*,widest=0}

\newcommand{\legendre}[2]{\left(\hspace{-1pt}\frac{#1}{#2}\hspace{-1pt}\right)}

\newtheorem{theorem}{Theorem}[section]
\newtheorem{lemma}[theorem]{Lemma}

\theoremstyle{definition}
\newtheorem{definition}{Definition}[section]
 \newtheorem{proposition}[theorem]{Proposition}
  
\newtheorem{example}{Example}[section]

\newtheorem{conjecture}[theorem]{\bf Conjecture}

\theoremstyle{remark}

\numberwithin{equation}{section}

\begin{document}

\title[Ramanujan tau function]{Representation of Ramanujan's tau function by twisted divisor functions}

%    Information for first author
\author{Tianyu Ni}
%%    Address of record for the research reported here
\address{School of Mathematical and Statistical Sciences,
Clemson University,
Clemson, SC 29634-0975,
USA}

%\email{tianyuni1994math@gmail.com}
%    \thanks will become a 1st page footnote.
%\thanks{The first author was supported in part by NSF Grant \#000000.}

%    Information for second author
%\author{Author Two}
%\address{Mathematical Research Section, School of Mathematical Sciences,
%Australian National University, Canberra ACT 2601, Australia}
%\email{two@maths.univ.edu.au}
%\thanks{Support information for the second author.}

%    General info

%\date{January 1, 2001 and, in revised form, June 22, 2001.}

%\dedicatory{This paper is dedicated to our advisors.}


\begin{abstract}
We present an infinite family of identities that represent Ramanujan's tau function in terms of convolution sums of twisted divisor functions. Our method involves explicitly constructing non-vanishing level $1$ cusp forms from modular forms of higher levels.
\end{abstract}

\maketitle
%\tableofconten

\section{Introduction}
The Ramanujan's tau function $\tau$ is defined by the following expression:
\begin{align*}
    \Delta(z):=q\prod_{n\geq1}(1-q^n)^{24}=\sum_{n\geq1}\tau(n)q^n,
\end{align*}
where $q=e^{2\pi iz}$ with $\im (z)>0$, and the function 
$\Delta (z)$ is a cusp form of weight $12$ and level $1$.
There are various works on finding a formula for $\tau(n)$ which involve convolution sums of the divisor functions, see \cite{Ramanujanid} and the references therein. In this paper, we prove an infinite family of identities that represent $\tau(n)$ in terms of convolution sums of twisted divisor functions that are defined below.  Throughout the paper, let $D\geq1$ denote an odd square-free integer unless specified.  
\begin{definition}\label{def:twsumdiv}Let $l\geq3$ be an integer, and let $\phi$ and $\psi$ be respectively Dirichlet characters 
modulo $D_1$ and $D_2$ such that  $D=D_1D_2$ and
$\phi(-1)\psi(-1)=(-1)^{l}$. We define 
    \begin{align*}
   \sigma_{l-1,\phi,\psi}(n):&=\begin{cases}
        -L(1-l,\phi)L(0,\psi)&n=0,\\
\sum\limits_{\substack{d_1,d_2>0\\d_1d_2=n}}\phi(d_1)\psi(d_2)d_1^{l-1}&n\geq1.
    \end{cases}
    \end{align*}
    \end{definition}
Note that $\sigma_{l-1,\phi,\psi}(n)$ appear as the Fourier coefficients of some Eisenstein series (see e.g., \cite[p.129]{diamondshurman}). 
To state our results concisely, we also  introduce the following notation, which arises from the Fourier coefficients of certain cusp forms of level $1$; see Proposition \ref{prop:fouriercoexplicit}.
\begin{definition}\label{def:fDlke}
  Let $\chi$ be a primitive Dirichlet character modulo $D$, and let $\ell,k$ be integers with the same parity such that $3\leq\ell\leq k$ and $\chi(-1)=(-1)^{\ell}$. For $n\geq1$, we define
    \begin{gather*}
        a_{D,\ell,k,e}(n;\chi):=\sum_{D=D_1D_2}\overline{\chi}_2(-1)D_2^{-e}\hspace{-5pt}\sum_{\substack{a_1,a_2\geq0\\a_1+a_2=nD_2}}\hspace{-5pt}\sigma_{\ell-1,\chi_1,\overline{\chi}_2}(a_1)\sigma_{k-1,\overline{\chi}_1,\chi_2}(a_2) \cdot c_{e,a_1,a_2},\\
        c_{e,a_1,a_2}:=\sum_{r=0}^e(-1)^ra_1^ra_2^{e-r}\binom{e+\ell-1}{e-r}\binom{e+k-1}{r},
    \end{gather*}
where the summation $\sum_{D=D_1D_2}$ is over all factorizations of $D=D_1D_2$ as a product of two positive integers,  and $\chi=\chi_1\chi_2$ is the corresponding factorization into primitive characters modulo $D_1$ and $D_2$. We write $a_{D,k,e}(n)$ for $a_{D,k,k,e}(n,\chi_D)$, where $\chi_D$ denotes the quadratic character modulo $D$.
%We write 
%\begin{align*}
%    \tilde{f}_{D,\ell,k,e}(n;\chi):=\frac{   f_{D,\ell,k,e}(n;\chi)}{   f_{D,\ell,k,e}(1;\chi)}.
%\end{align*}
   \end{definition} 
We can now state the main result of the paper.
    \begin{theorem}\label{thm:mainresultodd}Let  $\chi$ be a primitive Dirichlet character modulo $D$, and let
\begin{align*}
    \tilde{a}_{D,\ell,k,e}(n;\chi):=\frac{   a_{D,\ell,k,e}(n;\chi)}{   a_{D,\ell,k,e}(1;\chi)}.
\end{align*}
    \begin{enumerate}
        \item If $\chi(-1)=-1$, then  
        \begin{align}
                \tau(n)&=\tilde{a}_{D,3,5,2}(n;\chi),\tag{a.1}\\
                \tau(n)&=\tilde{a}_{D,3,7,1}(n;\chi).\tag{a.2}\label{eq:371odd}
        \end{align}
        \item  If $\chi(-1)=1$, then
        \begin{align}
               \tau(n)&=\tilde{a}_{D,4,6,1}(n;\chi).\tag{b}\label{eq:evennoncentral}
        \end{align}
        \item If $D=1$ or a prime congruent to $5\pmod8$, then
        \begin{align}
                \tau(n)&=\tilde{a}_{D,4,2}(n).\tag{c}\label{eq:evencentral}
        \end{align}
    \end{enumerate}
%  \begin{align*}
%\tau(n)&=
%\frac{\sum\limits_{D=D_1D_2}\overline{\chi}_2(-1)\sum\limits_{m=0}^{nD_2}\sigma_{2,\chi_1,\overline{\chi}_2}(m)\sigma_{4,\overline{\chi}_1,\chi_2}(nD_2-m)\left(\frac{45m^2}{D_2^2}-\frac{36mn}{D_2}+6n^2\right)}{\sum\limits_{D=D_1D_2}\overline{\chi}_2(-1)\sum\limits_{m=0}^{D_2}\sigma_{2,\chi_1,\overline{\chi}_2}(m)\sigma_{4,\overline{\chi}_1,\chi_2}(D_2-m)\left(\frac{45m^2}{D_2^2}-\frac{36m}{D_2}+6\right)}
%,\\\tau(n)&=\frac{\sum\limits_{D=D_1D_2}\overline{\chi}_2(-1)\sum\limits_{m=0}^{nD_2}\sigma_{2,\chi_1,\overline{\chi}_2}(m)\sigma_{6,\overline{\chi}_1,\chi_2}(nD_2-m)\left(3n-\frac{10m}{D_2}\right)}{\sum\limits_{D=D_1D_2}\overline{\chi}_2(-1)\sum\limits_{m=0}^{D_2}\sigma_{2,\chi_1,\overline{\chi}_2}(m)\sigma_{6,\overline{\chi}_1,\chi_2}(D_2-m)\left(3-\frac{10m}{D_2}\right)}. 
%  \end{align*}
  \end{theorem}

Before going on to sketch the idea of the proof, we  look at several examples. 

\begin{example}
First, we consider the case $D=1$, where $\sigma_{l-1,\phi,\psi}(n)$ becomes the divisor function $\sigma_{l-1}(n)=\sum_{d\mid n}d^{l-1}$ for $n\geq1$ and $\sigma_{l-1}(0)=\zeta(1-l)/2$. Let $\mathbf{1}$ denote the trivial character. Note that 
  \begin{align*}
      a_{1,4,6,1}(n;\mathbf{1})
%&=\sum_{\substack{a_1,a_2\geq0\\a_1+a_2=n}}\sigma_{3}(a_1)\sigma_{5}(a_2)\sum_{r=0}^1(-1)^ra_1^ra_2^{1-r}\binom{4}{1-r}\binom{6}{r}\\
&=\sum_{ m=0 }^n\sigma_{3}(m)\sigma_{5}(n-m)\sum_{0\leq r\leq1}(-1)^rm^r(n-m)^{1-r}\binom{4}{1-r}\binom{6}{r}
\\&=\frac{\zeta(-3)}{2}\sigma_5(n)4n-\frac{\zeta(-5)}{2}\sigma_3(n)6n+\sum\limits_{m=1}^{n-1}(4n-10m)\sigma_3(m)\sigma_5(n-m)\\&=\frac{1}{60}\sigma_5(n)n+\frac{1}{84}\sigma_{3}(n)n+\sum_{m=1}^{n-1}(4n-10m)\sigma_{3}(m)\sigma_{5}(n-m).
  \end{align*}
%  where $\sigma_3(0)=\frac{\zeta(-3)}{2}=\frac{1}{240}$ and $\sigma_5(0)=\frac{\zeta(-5)}{2}=-\frac{1}{504}$. 
Now, \eqref{eq:evennoncentral} becomes
\begin{align}
    \tau(n)&
    %=\frac{   f_{1,4,6,1}(n;\mathbf{1})}{f_{1,4,6,1}(1;\mathbf{1})}
    =\frac{5}{12}n\sigma_3(n)+\frac{7}{12}n\sigma_5(n)+70\sum_{m=1}^{n-1}(2n-5m)\sigma_{3}(m)\sigma_{5}(n-m),\tag{i}\label{eq:4,6,1level1}
\end{align}
which also appeared in \cite[(ix)]{Ramanujanid}. We can also check that
\begin{align*}
    a_{1,4,2}(n)
    %&=\sum_{m=0}^{n}\sigma_{3}(m)\sigma_{3}(n-m)\sum_{r=0}^2(-1)^rm^r(n-m)^{2-r}\binom{5}{2-r}\binom{5}{r}\\&
    =\frac{1}{12}n^2\sigma_{3}(n)+5\sum\limits^{n-1}_{m=1}(2n-3m)(n-3m)\sigma_{3}(m)\sigma_{3}(n-m).
\end{align*}
Hence \eqref{eq:evencentral} becomes
\begin{align}
    \tau(n)
    %=\frac{f_{1,4,2,2}(n,\mathbf{1})}{f_{1,4,2,2}(1,\mathbf{1})}
    =n^2\sigma_3(n)+60\sum\limits^{n-1}_{m=1}(2n-3m)(n-3m)\sigma_{3}(m)\sigma_{3}(n-m),\tag{ii}\label{eq:442level1}
\end{align}
as given in \cite[p.344]{transRC1966} and \cite[p. 428]{Lietheory}. Identities \eqref{eq:4,6,1level1} and \eqref{eq:442level1} are examples of the van der Pol type identities for Ramanujan’s tau function \cite{vanderpol1951}, which can be proved by considering Maass differential operators \cite{Maass} or Rankin-Cohen brackets \cite{Cohen'smodularformC_k} on Eisenstein series of level 1, see \cite{Lanphiertauid} and \cite{Ramanujanid} for details. 
\end{example}
\begin{example}

We then examine the cases when $D>1$, which are a generalization of the aforementioned classical identities from the perspective of levels rather than weights of modular forms.  Our formulas become complicated, as there will be more convolution sums involved. For example, we look at \eqref{eq:371odd} for $D=p$. Note that
\begin{align*}
a_{p,3,7,1}(n,\chi)
=&\sum_{p=D_1D_2}\overline{\chi}_2(-1)\sum_{m=0}^{nD_2}(3n-\frac{10m}{D_2})\sigma_{2,\chi_1,\overline{\chi}_2}(m)\sigma_{6,\overline{\chi}_1,\chi_2}(nD_2-m)\\
%In particular, if $D=p$ is a prime, then $ a_{p,3,7,1}(n,\chi)=$
%\begin{align*}
=&-\sum_{m=0}^{np}(3n-\frac{10m}{p})\sigma_{2,\mathbf{1},\overline{\chi}}(m)\sigma_{6, \mathbf{1},\chi}(np-m)\\&+\sum_{m=0}^n(3n-10m)\sigma_{2,\chi,\mathbf{1}}(m)\sigma_{6,\overline{\chi},\mathbf{1}}(n-m).
\end{align*}
Here, $\sigma_{l-1,\phi,\psi}(0)=0$ unless $D_2=1$, in which case $\sigma_{l-1,\phi,\mathbf{1}}(0)=L(1-l,\phi)/2$ and we write $\sigma_{l-1,\phi}(n)$ for $\sigma_{l-1,\phi,\mathbf{1}}(n)$.  Note also that for $n\geq1$ we have $$\sigma_{l-1,\mathbf{1},\psi}(n)=\sum_{d\mid n}\psi(n/d)d^{l-1}=n^{l-1}\sigma_{1-l,\psi}(n),$$
where $\sigma_{1-l,\psi}(n):=\sum_{m\mid n}\psi(m)m^{1-l}$. For $n>1$  we have
\begin{align*}
  a_{p,3,7,1}(n,\chi)= &-\sum_{m=1}^{np-1}m^2(3n-\frac{10m}{p})(np-m)^6\sigma_{-2,\overline{\chi}}(m)\sigma_{-6,\chi}(np-m)\\&+\frac{L(-2,\chi)}{2}\sigma_{6,\overline{\chi}}(n)3n-\frac{L(-6,\overline{\chi})}{2}\sigma_{2,\chi}(n)7n\\&+\sum_{m=1}^{n-1}(3n-10m)\sigma_{2,\chi}(m)\sigma_{6,\overline{\chi}}(n-m),\end{align*}
 and for $n=1$ we have   
\begin{align}a_{p,3,7,1}(1,\chi)=&-\sum_{m=1}^{p-1}m^2(3-\frac{10m}{p})(p-m)^6\sigma_{-2,\overline{\chi}}(m)\sigma_{-6,\chi}(p-m)\nonumber\\&+\frac{3}{2}{L(-2,\chi)}-\frac{7}{2}L(-6,\overline{\chi}). \label{eq:1stcop371}
\end{align}
Note that, compared to the case $D=1$, our expression contains one additional term of the convolution sum of twisted divisor functions, which is related to another parameter $p$. We now provide some numerical examples for $p=7$; see \cite{codefortwisted}. Let $w=e^{\pi i/3}$. There are three odd primitive Dirichlet characters modulo $7$:
\begin{center}
    \begin{tabular}{c|r r c c c c}
        $n$ & -1 & 1 & 2 & 3 & 4 & 5 \\
        \hline
$\chi_7$   & -1 & $1$ & $1$ & $-1$ & $1$ & $-1$ \\
        
      $\varphi$   & -1 & $1$ & $w^2$ & $w$ & $-w$ & $-w^2$ \\
      $\overline{\varphi}$   & -1 & $1$ & $-w$ & $-w^2$ & $w^2$ & $w$ 
    \end{tabular}
\end{center}
The first few values of $a_{7,3,7,1}(n,\chi)$ for $\chi=\chi_7, \varphi$ and $\overline{\varphi}$
%$a_{7,3,7,1}(n,\chi_7)$, $a_{7,3,7,1}(n,\varphi)$ and $a_{7,3,7,1}(n,\overline{\varphi})$ 
are given in the following tables, where one can easily verify that $\tau(n)=\tilde{a}_{7,3,7,1}(n;\chi)$.
\begin{center}
\begin{tabular}{ r  r }

 $n$ & $\tau(n)$  \\ 
 \hline
 $1$ & $1$  \\  
 $2$ & $-24$\\
 $3$ & $252$\\
 $4$ & $-1472$\\
 $5$ & $4830$\\
 $6$ & $-6048$\\
 $7$ & $-16744$ \\
 $8$ & $84480$\\
 $9$ & $-113643$\\
 $10$ & $-115920$
\end{tabular}
\quad
\begin{tabular}{ r  r }
$n$ & $a_{7,3,7,1}(n,\chi_7)$  \\ 
 \hline
  $1$ &   $\frac{66816}{7}$  \\
 $2$ & $\frac{-1603584}{7}$\\  
 $3$ &  $2405376$\\
 $4$ & $ \frac{-98353152}{7}$\\
 $5$ & $46103040$\\
 $6$ & $  -57729024$\\
 $7$ & $ -159823872$\\
 $8$ & $\frac{5644615680}{7}$ \\
 $9$ & $ \frac{-7593170688}{7}$\\
 $10$ & $-1106472960$\\
\end{tabular}
\\
\begin{tabular}{ r  r }
$n$ & $a_{7,3,7,1}(n,\varphi)$  \\ 
 \hline
  $1$ &      $\frac{-93600}{7}w + \frac{76896}{7}$\\
 $2$ & $\frac{2246400}{7}w - \frac{1845504}{7}$\\ 
 $3$ &  $-3369600 w + 2768256$\\
 $4$ & $\frac{137779200}{7}w - \frac{113190912}{7}$\\
 $5$ & $-64584000w + 53058240$\\
 $6$ & $80870400w - 66438144$\\
 $7$ & $223891200w - 183935232$\\
 $8$ &  $\frac{-7907328000}{7}w + \frac{6496174080}{7}$\\
 $9$ & $\frac{10636984800}{7}w - \frac{8738692128}{7}$\\
 $10$ & $1550016000w - 1273397760$\\
\end{tabular}\quad\begin{tabular}{ r  r }
$n$ & $a_{7,3,7,1}(n,\overline{\varphi})$  \\ 
 \hline
  $1$ & $\frac{93600}{7}w - \frac{16704}{7}$     \\
 $2$ &  $\frac{-2246400}{7}w + \frac{400896}{7}$\\ 
 $3$ & $3369600w - 601344$ \\
 $4$ & $\frac{-137779200}{7}w + \frac{24588288}{7}$\\
 $5$ & $64584000w - 11525760$\\
 $6$ & $ -80870400w + 14432256$\\
 $7$ & $-223891200w + 39955968$\\
 $8$ & $\frac{7907328000}{7}w - \frac{1411153920}{7}$\\
 $9$ & $\frac{-10636984800}{7}w + \frac{1898292672}{7}$\\
 $10$ & $-1550016000w + 276618240$\\
\end{tabular}
\end{center}
\end{example}
\begin{example}
    It is perhaps worth mentioning one application of Theorem \ref{thm:mainresultodd} to evaluate the value of the Dirichlet $L$-function.  Recall the following classical result (see e.g., \cite[Theorem 4.2]{cyclotomicWashington}). Let $\chi$ be a Dirichlet character $\chi$ modulo $q$, and let $m$ be a positive integer. Then 
    \begin{align*}
        L(1-m,\chi)=-\frac{q^{m-1}}{m}\sum_{a=1}^{q}\chi(a)B_{m}(\frac{a}{q}),
    \end{align*}
    where $B_m(x)$ is the Bernoulli polynomial defined by
    \begin{align*}
        \frac{te^{xt}}{e^t-1}=\sum_{n=0}^{\infty}B_n(x)\frac{t^n}{n!}.
    \end{align*}For a prime $p\equiv5\pmod8$ and $\chi_p(n):=
    \legendre{n}{p},$ we use Theorem \ref{thm:mainresultodd} \eqref{eq:evencentral} to give an alternative formula for $L(-3,\chi_p)$, which does not involve Bernoulli polynomials. One can easily verify that
        \begin{align*}
     a_{p,4,2}(n)
     = &\sum_{p=D_1D_2}\sum_{m=0}^{nD_2}(\frac{45 m^2}{D_2^2}-\frac{45mn}{D_2}+10n^2)\sigma_{3,\chi_1,\chi_2}(m)\sigma_{3,\chi_1,\chi_2}(nD_2-m)
     \\=&\sum_{m=1}^{np-1}(\frac{45 m^2}{p^2}-\frac{45mn}{p}+10n^2)\sigma_{3,\mathbf{1},\chi_p}(m)\sigma_{3,\mathbf{1},\chi_p}(np-m)\\&+\sum_{m=1}^{n-1}(45 m^2-45mn+10n^2)\sigma_{3,\chi_p}(m)\sigma_{3,\chi_p}(n-m)\\&+2\cdot\frac{L(-3,\chi_p)}{2}\cdot 10n^2\sigma_{3,\chi_p}(n).
    \end{align*} 
In particular, we have
\begin{align*}
      a_{p,4,2}(1)= &\sum_{m=1}^{p-1}(\frac{45 m^2}{p^2}-\frac{45m}{p}+10)\sigma_{3,\mathbf{1},\chi_p}(m)\sigma_{3,\mathbf{1},\chi_p}(p-m)+10L(-3,\chi_p) .\end{align*}  
Note that $\sigma_{3,\chi_p}(2)=1+2^{3}\legendre{2}{p}=-7$ and  $a_{p,4,2}(2)=$
    \begin{align*}     
  &-5-280L(-3,\chi_p)+\sum_{m=1}^{2p-1}(\frac{45 m^2}{p^2}-\frac{90m}{p}+40)\sigma_{3,\mathbf{1},\chi_p}(m)\sigma_{3,\mathbf{1},\chi_p}(2p-m).
\end{align*}
%Note again that $\sigma_{3,\mathbf{1},\chi_p}(m)=m^3\sigma_{-3,\chi_p}(m)$. 
Since $\tau(2)=-24$, \eqref{eq:evencentral} gives 
$$a_{p,4,2}(2)=-24a_{p,4,2}(1),$$ from which we deduce that
\begin{align*}
    L(-3,\chi_p)=&-\frac{1}{8}+\frac{1}{40}\sum_{m=1}^{2p-1}(\frac{45m^2}{p^2}-\frac{90m}{p}+40)\sigma_{3,\mathbf{1},\chi_p}(m)\sigma_{3,\mathbf{1},\chi_p}(2p-m)\\&+\frac{3}{5}\sum_{m=1}^{p-1}(\frac{45 m^2}{p^2}-\frac{45m}{p}+10)\sigma_{3,\mathbf{1},\chi_p}(m)\sigma_{3,\mathbf{1},\chi_p}(p-m).
\end{align*}
\end{example}
We now sketch the idea of the proof. Our method involves constructing a family of cusp forms of level $1$ from forms of higher levels. To do so, we take the Rankin-Cohen brackets of Eisenstein series of higher levels and then take the trace of them, as seen in Section \ref{sect;construction}. However, unlike the case of level $1$, where one can easily show that the construction is non-vanishing by checking that at least one Fourier coefficient is non-zero, our construction involves an additional parameter $D$, for which we are unable to show that it is non-vanishing directly. In Section \ref{sect:nonvofconstruction}, we show that our construction is non-vanishing by the Rankin-Selberg method. In Section \ref{sect:Fourier}, we prove our result by explicitly computing the Fourier coefficients. Finally, we conclude the paper by proposing two conjectures.

\section{Constructing level one cusp forms from higher level forms}\label{sect;construction}
In this section, we construct a family of cusp forms of level 1 from modular forms of higher levels. Note that this type of construction has already appeared in \cite{Kohnen-Zagier1981}, which was generalized in \cite{Kayath_Lane_Neifeld_Ni_Xue_2025} and \cite{twistedNagoya}. 
For $N\geq1$ and a Dirichlet character $\chi$ modulo $N$, let $M_k(N,\chi)$ and $S_k(N,\chi)$ be the space of modular forms and cusp forms of weight $k$, level $N$ and nebentypus $\chi$, respectively. We write $M_k(N)$ and $S_k(N)$ if $\chi$ is trivial, and use $M_k$ and $S_k$ to denote $M_k(1)$ and $S_k(1)$ when $N=1$. 

The Rankin–Cohen bracket of two modular forms \cite{Rankin1956,Cohen'smodularformC_k}, generalizing the product of two modular forms, produces a cusp form. 

\begin{definition}\label{def:rankincohen}
Let $f(z)\in M_{a}(N,\nu_1)$ and $g(z)\in M_{b}(N,\nu_2)$ be modular forms of level $N$, weights $a$ and $b$ and nebentypus $\nu_1$ and $\nu_2$, respectively. For an integer $e\geq1$, the $e$-th Rankin-Cohen bracket is given by
\begin{align}
    [f(z),g(z)]_e := \sum_{r=0}^e (-1)^r\binom{e+a-1}{e-r}\binom{e+b-1}{r}f(z)^{(r)}g(z)^{(e-r)},\label{eq:defofrankin-cohenbracket}
\end{align}
where $f(z)^{(r)}$ is the $r$-th normalized derivative $f(z)^{(r)}:=\frac{1}{(2\pi i)^r}\frac{d^r f(z)}{dz^r}$ of $f$.  Moreover, $[f,g]_e\in S_{a+b+2e}(N,\nu_1\nu_2)$; see \cite[Theorem 7.1]{Cohen'smodularformC_k}. 
Note that the definition in Zagier \cite[(73)]{Zagier1976} is related to \eqref{eq:defofrankin-cohenbracket} through $F_{e}^{(a,b)}(f(z),g(z))= (-2\pi i)^e e![f(z),g(z)]_e$.
\end{definition}
We introduce a type of Eisenstein series; see also \cite[pp. 185]{Kohnen-Zagier1981} and \cite[pp. 269-270]{MiyakeMFbook}.
Assume $k\geq3$. Let $D\geq1$ be an odd square-free integer and $\chi$ be a primitive character mod $D$ such that $\chi(-1)=(-1)^{k}$. Define
\begin{align*}
    G_{k,\chi}(z):&=\sum_{n=0}^{\infty}\sigma_{k-1,\chi}(n)q^n\in M_k(D,\chi),\\
    \sigma_{k-1,\chi}(n):&=\begin{cases}
        \frac{L(1-k,\chi)}{2}&n=0,\\
        \sum\limits_{d\mid n}\chi(d)d^{k-1}&n\geq1.   
    \end{cases}
\end{align*}
Here, $L(s,\chi)=\sum_{n\geq1}\chi(n)n^{-s}$ ($\re(s)\gg0$) is the Dirichlet series.

Now, we construct a family of cusp forms of level one by taking the trace of Rankin-Cohen brackets of the aforementioned Eisenstein series. 
\begin{definition}\label{def:construction}Let $\chi$ be a primitive character mod $D$, and let $\chi_D$ denote the quadratic character mod $D$. For $e\geq1$ and integers $\ell,k$ with the same parity, such that $3\leq\ell\leq k$ and $\chi(-1)=(-1)^{\ell}$, we define  
\begin{align}
\mathcal{F}_{D,\ell,k,e}(z;\chi):=\begin{cases}\Tr^D_1[G_{\ell,\chi}(z),G_{k,\overline{\chi}}(z)]_e &\ell<k,\\\Tr^D_1[G_{k,\chi_D}(z),G_{k,\chi_D}(z)]_e &\ell=k,\end{cases}\label{eq:defoffancyG}
\end{align}
where $\Tr_1^D$  is the trace map 
%(see e.g. \cite[p.271]{GZ-86}) 
\begin{align*} 
    \Tr_1^D: M_{m}(D)\rightarrow M_m(1),\quad g\mapsto \sum_{\gamma\in\Gamma_0(D)\backslash\Gamma_0(1)}g|_m\gamma,
\end{align*}
and for $m\in\mathbb{R}$ and $\gamma=\begin{bsmallmatrix}
    a&b\\c&d
\end{bsmallmatrix}\in \Gl_2^+(\mathbb{R})$  the slash operator \cite[Theorem 7.1]{Cohen'smodularformC_k} is
\begin{align} 
g(z)|_m\gamma=\det(\gamma)^{m/2} (cz+d)^{-m} g\left(\frac{az+b}{cz+d}\right).\label{eq:slash}
\end{align}
We write $\mathcal{F}_{D,k,e}(z)$ for $\mathcal{F}_{D,k,k,e}(z;\chi_D)$.
%; see also \cite[(1.11)]{Kayath_Lane_Neifeld_Ni_Xue_2025}.
\end{definition}

%Note that $\mathcal{F}_{D,\ell,k,e}(z;\chi)$ is a cusp form of weight $k+\ell+2e$ and level one.
\section{Non-vanishing of the construction}\label{sect:nonvofconstruction}
In this section, we show that $\mathcal{F}_{D,\ell,k,e}$ is not identically zero in several cases. The idea is to use the Rankin-Selberg method (see e.g., \cite{Rankin-Selbergsurvey} for an exposition) to transform the problem into the non-vanishing of twisted $L$-values. 
For  $f,g\in M_k(N,\chi)$ such that $fg$ is a cusp form, the Petersson inner product is given by 
\begin{align*}
    \langle f,g\rangle_{N}=\int_{\Gamma_{0}(N)\backslash\mathbb{H}}f(z)\overline{g(z)}y^k \frac{dxdy}{y^2}.
\end{align*}
We write $\Gamma_{\infty}:=\{\pm\begin{psmallmatrix}
    1& n\\ 0&1
\end{psmallmatrix}:n\in\mathbb{Z}\}$. Let $E_{k,N}^{\ast}(z;\chi)$ be the Eisenstein series for the cusp at infinity of level $N$ \cite[p.~272]{MiyakeMFbook}:
\begin{align*}
    E_{k,N}^{\ast}(z;\chi):=\sum_{\begin{psmallmatrix}a&b\\c&d\end{psmallmatrix}\in\Gamma_{\infty}\backslash\Gamma_0(N)}\frac{\chi(d)}{(cz+d)^k}.
\end{align*}

We first consider the case $\ell<k$.
Recall the classical result on the Rankin-Selberg method below, which was reformulated and generalized in Zagier \cite{Zagier1976}, keeping in mind the difference between \eqref{eq:defofrankin-cohenbracket} and the one used therein.
\begin{lemma}[{\cite[Proposition 6]{Zagier1976}}]\label{lem:RankinSelbergZagier}
Let $k_1$ and $k_2$ be real numbers with $k_2\geq k_1+2>2$. Let $f(z)=\sum_{n\geq1}a(n)q^n$ and $g(z)=\sum_{n\ge0}b(n)q^n$ be modular forms in $S_k(N,\chi)$ and $M_{k_1}(N,\chi_1)$, where $k=k_1+k_2+2e, e\geq0$ and $\chi=\chi_1\chi_2$. 
Then 
    \begin{align}
        \langle f,[g,E^*_{k_2,N}(\cdot;\chi_2)]_e\rangle_{N}=\frac{(-1)^e}{e!}\frac{\Gamma(k-1)\Gamma(k_2+e)}{(4\pi)^{k-1}\Gamma(k_2)}\sum_{n=1}^{\infty}\frac{a(n)\overline{b(n)}}{n^{k_1+k_2+e-1}}.\label{eq:RankinSelberg}
    \end{align}
\end{lemma}
For a normalized Hecke eigenform $f(z)=\sum_{n\geq1}a_f(n)q^n$ in $S_K$, its twisted Hecke $L$-series is given by $$L(f,\chi,s):=\sum_{n\geq1}a_f(n)\chi(n)n^{-s},$$  which is holomorphic in the reigion $\re(s)>1+\frac{K-1}{2}$ (due to Deligne's bound). A standard computation gives the following result; see also \cite[(72)]{Zagier1976}.
\begin{lemma}[{\cite[Lemma 2.5]{twistedNagoya}}]\label{lem:easycomp}
  Let $f(z)=\sum_{n\geq1}a_{f}(n)q^n\in S_K$ be a normalized Hecke eigenform. Then for a Dirichlet character $\chi$, an integer $\ell\geq3$ and a complex number $s$ with $\re(s)>\ell+\frac{K-1}{2}$, we have 
  \begin{align*} 
    \sum_{n=1}^{\infty}\frac{a_f(n)\sigma_{\ell-1,\chi}(n)}{n^{s}}=\frac{L(f,s)L(f,\chi,s-\ell+1)}{L(2s-\ell+2-K,\chi)}.
  \end{align*}
\end{lemma}
The following proposition is likely well-known to experts, but we include a proof here due to the lack of an exact reference.
\begin{proposition}\label{prop:rsrc}
Suppose $\ell<k$. If $f\in S_{k+\ell+2e}$ is a normalized Hecke eigenform, then 
    \begin{align*}
        \langle f, \mathcal{F}_{D,\ell,k,e}\rangle_1=&
        %\frac{(-1)^e}{e!}\frac{\Gamma(k+\ell+2e-1)\Gamma(k+e)D^kL(k,\chi)}{(4\pi)^{k+\ell+2e-1}(-2\pi i)^k\overline{\tau(\chi)}L(k,\overline{\chi})}\\&\quad\quad\quad\times 
        \tilde{c}\cdot L(f,k+\ell+e-1)L(f,\overline{\chi},k+e),\nonumber
    \end{align*}
    where $\tilde{c}$ is some non-zero constant depending on $k,\ell,e$ and $\chi$.
\end{proposition}
\begin{proof}
    Since $\langle f,g\rangle_{N}=\langle f,\Tr^D_1 g\rangle_1$ for $f\in S_K$ and $g\in M_K(D)$ (see e.g., \cite[p.271]{GZ-86}), we have
    \begin{align}
        \langle f,\mathcal{F}_{D,k,\ell,e}\rangle_1=\langle f,[G_{\ell,\chi}(z),G_{k,\overline{\chi}}(z)]_e\rangle_D.\label{eq:inprodtr}
    \end{align}
    Note that (see e.g., \cite[Lemma 2.2]{twistedNagoya}):
\begin{align}
    G_{k,\overline{\chi}}(z)=\frac{(k-1)!D^kL(k,\chi)}{(-2\pi i)^k\tau(\chi)}E_{k,D}^{\ast}(z;\chi),\label{eq:relationbetweentwoeisenstein}
\end{align}
where $\tau(\chi)$ is the Gauss sum.  By Lemma \ref{lem:RankinSelbergZagier}, \eqref{eq:inprodtr} and \eqref{eq:relationbetweentwoeisenstein}, we get
\begin{align}
    \langle f,\mathcal{F}_{D,\ell,k,e}\rangle_1=
    %\frac{(-1)^e\Gamma(k+\ell+2e-1)\Gamma(k+e)D^kL(k,\chi)}{e!(4\pi)^{k+\ell+2e-1}(-2\pi i)^k\overline{\tau(\chi)}}
    c\cdot\sum_{n=1}^{\infty}\frac{a_f(n)\sigma_{\ell-1,\overline{\chi}}(n)}{n^{k+\ell+e-1}}.\label{eq:midstep}
\end{align}
Here, $c$ is some non-zero constant depending only on $\ell,k,e$ and $\chi$. Note that for $K=k+\ell+2e$ and $s=k+\ell+e-1$, we have 
$\re(s)-(\ell+\frac{K-1}{2})=\frac{k-\ell-1}{2}>0$ since $k-\ell\geq2$. Hence, from Lemma \ref{lem:easycomp} we deduce that
\begin{align}
\sum_{n=1}^{\infty}\frac{a_f(n)\sigma_{\ell-1,\overline{\chi}}(n)}{n^{k+\ell+e-1}}=\frac{L(f,k+\ell+e-1)L(f,\overline{\chi},k+e)}{L(k,\overline{\chi})}.\label{eq:twcvsum}
\end{align}
Now, the result follows from \eqref{eq:midstep} and  \eqref{eq:twcvsum}. 
\end{proof}

\begin{proposition}\label{prop:nonval<k}
Suppose $\ell<k$. If $\dim S_{k+\ell+2e}\geq1$ then $\mathcal{F}_{D,\ell,k,e}$ is not identically  zero.
\end{proposition}
\begin{proof}
    Suppose $\mathcal{F}_{D,\ell,k,e}\equiv0$. Then $\langle f,\mathcal{F}_{D,\ell,k,e}\rangle_1=0$ for all normalized Hecke eigenforms $f$ in $S_{k+\ell+2e}$. 
     On the other hand, by Proposition \ref{prop:rsrc} and the fact that $k+\ell+e-1$ and $k+e$ are respectively within the regions of absolute convergence of $L(f,s)$ and $L(f,\overline{\chi},s)$, we get $\langle f,\mathcal{F}_{D,\ell,k,e}\rangle_1\neq0$, leading to a contradiction.
\end{proof}
For $n\geq1$, let $a_{D,\ell,k,e}(n;\chi)$ denote the $q^n$-coefficient of $\mathcal{F}_{D,\ell,k,e}(z;\chi)$, and let 
\begin{align}
    \tilde{\mathcal{F}}_{D,\ell,k,e}(z;\chi):=\frac{\mathcal{F}_{D,\ell,k,e}(z;\chi)}{a_{D,\ell,k,e}(1;\chi)},
\end{align}
where we understand that $a_{D,\ell,k,e}(1,\chi)$ has to be nonzero to validate the definition of $\tilde{\mathcal{F}}_{D,\ell,k,e}(z;\chi)$. See Conjecture \ref{conj:nonvans=ishingof1stco} for the non-vanishing of $a_{D,\ell,k,e}(1,\chi)$.
\begin{proposition}\label{prop:k=12l<k}
    Let $\chi$ be a primitive Dirichlet character mod $D$. Then
    \begin{align*}
        \Delta(z)=\begin{cases}
            \tilde{\mathcal{F}}_{D,3,5,2}(z;\chi)=\tilde{\mathcal{F}}_{D,3,7,1}(z;\chi)&\chi(-1)=-1,\\\tilde{\mathcal{F}}_{D,4,6,1}(z;\chi)&\chi(-1)=1.
        \end{cases}
    \end{align*}
\end{proposition}
\begin{proof}
   It is immediate from Proposition \ref{prop:nonval<k} and the fact that $\Delta(z)$ is the unique normalized Hecke eigenform in $S_{12}$.
\end{proof}
We now consider the case $\ell=k$, where we can not use Lemma \ref{lem:RankinSelbergZagier} to compute  $\langle f,\mathcal{F}_{D,k,e}\rangle_1$. To tackle this difficulty, we can follow Shimura \cite{Shimura1976} and Lanphier \cite{ShimuraoperatorLanphier}; see \cite[pp. 21-23]{Kayath_Lane_Neifeld_Ni_Xue_2025} for complete details. Here we only state the result. 
\begin{proposition}[{\cite[Proposition 5.6]{Kayath_Lane_Neifeld_Ni_Xue_2025}}]\label{prop:rscentrallvalue}
Suppose $e\equiv0\pmod2$, $k\ge4$ and $\chi_D(-1)=(-1)^k$. If $f\in S_{2k+2e}$ is a normalized Hecke eigenform, then
\begin{align*}
    \langle f,\mathcal{F}_{D,k,e}\rangle_1=\frac{\Gamma(2k+2e-1)\Gamma(k+e)L(1-k,\chi_D)}{2e!(4\pi)^{2k+2e-1}\Gamma(k)L(k,\chi_D)}L(f,2k+e-1)L(f,\chi_D,k+e).
\end{align*}
\end{proposition}
Now, the non-vanishing of $\mathcal{F}_{D,k,e}$ comes from the following result on the non-vanishing of the twisted central $L$-value due to Kohnen-Zagier \cite{Kohnen-Zagier1981}.
\begin{proposition}[{\cite[Corollary 2]{Kohnen-Zagier1981}}]\label{prop:kzcentralLspecialcase}
    For $k\equiv2\pmod4$ and $D$ a prime congruent to $5\pmod8$, there is at least one Hecke eigenform $f\in S_{2k}$ for which both $L(f,k)$ and $L(f,\chi_D,k)$ are different from zero.
\end{proposition}
\begin{proposition}\label{prop:nonvanl=k}
    Let $k\geq3$ and $e>0$ be integers such that $e\equiv0\pmod2$ and $k+e\equiv2\pmod4$, and $D$ be a prime congruent to $5\pmod8$. Then $\mathcal{F}_{D,k,e}$ is not identically zero.
\end{proposition}
\begin{proof}
    Suppose $\mathcal{F}_{D,k,e}\equiv0$. Then $\langle f,\mathcal{F}_{D,k,e}\rangle_1=0$ for all normalized Hecke eigenforms $f$ in $S_{2k+2e}$, implying that $L(f,\chi_D,k+e)=0$ by Proposition \ref{prop:rscentrallvalue}, which contradicts Proposition \ref{prop:kzcentralLspecialcase}.
\end{proof}
\begin{proposition}\label{prop:k=12l=k}
    Let $D=1$ or a prime congruent to $5\pmod8$. Then
    \begin{align*}
        \Delta(z)=\tilde{\mathcal{F}}_{D,4,2}(z;\chi_D).
    \end{align*}
\end{proposition}
    \begin{proof}
      The case $D=1$ is trivial.  The other case is immediate from Proposition \ref{prop:nonvanl=k} and the fact that $\Delta(z)$ is the unique normalized Hecke eigenform in $S_{12}$.
    \end{proof}

%\begin{align*}
%g_{D,3,5,2}(n;\chi)&=\hspace{-5pt}\sum_{D=D_1D_2}\overline{\chi}_2(-1)D_2^{-2}\sum_{m=0}^{nD_2}\sigma_{2,\chi_1,\overline{\chi}_2}(m)\sigma_{4,\overline{\chi}_1,\chi_2}(nD_2-m)c_{2,m,nD_2-m},\\
%c_{2,m,nD_2-m}&=\sum_{r=0}^2(-1)^rm^r(nD_2-m)^{2-r}\binom{4}{2-r}\binom{6}{r}\\&=45m^2-36D_2mn+6D_2^2n^2,
%\end{align*}
%which implies that $g_{D,3,5,2}(n;\chi)=$
%\begin{align*}\sum\limits_{D=D_1D_2}\overline{\chi}_2(-1)\sum\limits_{m=0}^{nD_2}\sigma_{2,\chi_1,\overline{\chi}_2}(m)\sigma_{4,\overline{\chi}_1,\chi_2}(nD_2-m)\left(\frac{45m^2}{D_2^2}-\frac{36mn}{D_2}+6n^2\right).
%\end{align*}

\section{Computation of the Fourier coefficients}\label{sect:Fourier}
In this section, we compute the Fourier coefficients $a_{D,\ell,k,e}(n;\chi)$. We introduce another type of Eisenstein series; see also \cite[p.193]{Kohnen-Zagier1981} and \cite[pp. 269-270]{MiyakeMFbook}. Let $D=D_1D_2$ such that $D_1>0$, and $\chi=\chi_1\chi_2$, where $\chi_1$ and $\chi_2$ are primitive Dirichlet characters mod $D_1$ and mod $D_2$, respectively. Define 
\begin{align*}
    G_{k,\chi_1,\chi_2}(z):&=\sum_{n\geq0}\sigma_{k-1,\chi_1,\chi_2}(n)q^n,\nonumber\\
    \sigma_{k-1,\chi_1,\chi_2}(n):&=\begin{cases}
        -L(1-k,\chi_1)L(0,\chi_2)&n=0,\\
\sum\limits_{\substack{d_1,d_2>0\\d_1d_2=n}}\chi_1(d_1)\chi_2(d_2)d_1^{k-1}&n\geq1.
    \end{cases}
    %\label{eq:defoftwistsumchi1chi2}
\end{align*}
The following lemma is useful for computing the trace; see also \cite[pp.273-275]{GZ-86}. 
\begin{lemma}[{\cite[Lemma 3.1]{twistedNagoya}}]\label{lem:FourierexpansionofGkD} Let $D=D_1D_2$ with $D_1>0$
and $\chi$ be a primitive Dirichlet character mod $D$ such that $\chi(-1)=(-1)^k$. If $\begin{psmallmatrix}
    a&b\\c&d
\end{psmallmatrix}$ is a matrix in $\Sl_2(\mathbb{Z})$ such that $\gcd(c,D)=D_1$ then
\begin{align}
    G_{k,\chi}(z)\bigg|_k\begin{pmatrix}
        a & b\\ c& d\end{pmatrix}=\chi_2(c)\chi_1(d)\overline{\chi}_{1}(D_2)\overline{\chi}_{2}(D_1)\frac{G(\overline{\chi}_{1})}{G(\overline{\chi})}G_{k,\chi_1,\overline{\chi}_{2}}\left(\frac{z+c^{\ast}d}{D_2}\right),
\end{align}
where $\chi=\chi_1\chi_2$ is the product of primitive characters $\chi_1$ mod $D_1$ and $\chi_2$ mod $D_2$, and $c^{\ast}$ is an integer with $cc^{\ast}\equiv 1\pmod{D_2}$ and $D_1\mid c^{\ast}$. 
\end{lemma}
For $m\geq1$ and $f(z)=\sum_{n\geq0}a_f(n)q^n\in M_{k}(N,\chi)$ we define the $U$-operator
\begin{align}
    U_mf(z)=\frac{1}{m}\sum_{v~{\rm mod}~m}f\left(\frac{z+v}{m}\right)=\sum_{n\geq0}a_f(mn)q^n.\label{eq:defofUmap}
\end{align}
Equivalently, we may write 
\begin{align*}U_mf(z)=m^{k/2-1}\sum_{v\text{ mod }m}f(z)\bigg|_k\begin{pmatrix}
    1 & v\\ 0& m
\end{pmatrix}.\label{eq:defofUslash}\end{align*} 
We are ready to compute $a_{D,\ell,k,e}(n;\chi)$. The special case $e=0$ and $k=\ell$ was computed by Kohnen-Zagier \cite[p.193]{Zagier1976}, and some other special cases have been calculated in \cite{Kayath_Lane_Neifeld_Ni_Xue_2025,twistedNagoya}. 
\begin{proposition}\label{prop:fourierco}
    Let $\mathcal{F}_{D,\ell,k,e}(z;\chi)$ be defined as in \eqref{eq:defoffancyG}. Then 
    \begin{align*}
        \mathcal{F}_{D,\ell,k,e}(z;\chi)=\sum_{D=D_1D_2}\overline{\chi}_2(-1)D_2^{-e}U_{D_2}\left([G_{\ell,\chi_1,\overline{\chi}_2}(z),G_{k,\overline{\chi}_1,\chi_2}(z)]_e\right).
    \end{align*}
   % where  the summation is over all factorizations of $D=D_1D_2$ as a product of two positive integers, and $\chi=\chi_1\chi_2$ is the product of primitive characters $\chi_1$ mod $D_1$ and $\chi_2$ mod $D_2$.
\end{proposition}
\begin{proof}
    Consider the following system of representatives  of $\Gamma_0(D)\backslash \Sl_{2}(\mathbb{Z})$ (see \cite[p.~276]{GZ-86} and \cite[Lemma 3.1]{Kayath_Lane_Neifeld_Ni_Xue_2025}):
    \begin{equation*}
        \left\{\begin{pmatrix}
            1 & 0 \\ D_1 & 1
        \end{pmatrix}\begin{pmatrix}
            1 & \mu \\ 0 & 1
        \end{pmatrix}~:~ D=D_1D_2,~\mu \text{ mod }D_2\right\}.
    \end{equation*}
By Lemma \ref{lem:FourierexpansionofGkD} and the fact that the slash operator \eqref{eq:slash} commutes with the Rankin-Cohen bracket, we have $   \mathcal{F}_{D,\ell,k,e}(z;\chi)=$
\begin{align*}
 &\sum_{D_1D_2=D}\sum_{\mu \text{ mod } D_2} \left[G_{\ell,\chi}(z),G_{k,\overline{\chi}}(z)\right]_e\bigg|_{k+\ell+2e}\begin{pmatrix}
            1 & 0 \\ D_1 & 1
        \end{pmatrix}\begin{pmatrix}
            1 & \mu \\ 0 & 1
        \end{pmatrix} \\
        =&\sum_{D_1D_2=D}\sum_{\mu \text{ mod } D_2}\left[G_{\ell,\chi}(z)\bigg|_{\ell}\begin{pmatrix}
            1 & 0 \\ D_1 & 1
\end{pmatrix}\begin{pmatrix}
            1 & \mu \\ 0 & 1
        \end{pmatrix},G_{k,\overline{\chi}}(z)\bigg|_{k}\begin{pmatrix}
            1 & 0 \\ D_1 & 1
        \end{pmatrix}\begin{pmatrix}
            1 & \mu \\ 0 & 1
        \end{pmatrix}\right]_e \\
        =&\sum_{D_1D_2=D}\sum_{\mu \text{ mod } D_2}\hspace{-5pt}\overline{\chi}_2(-1)D_2^{-1}\left[G_{\ell,\chi_1,\overline{\chi}_{2}}\left(\frac{z+\mu+D_1^*}{D_2}\right),G_{k,\overline{\chi}_{1},\chi_2}\left(\frac{z+\mu+D_1^*}{D_2}\right)\right]_e,
\end{align*}
where we used $G(\chi_{1})G(\overline{\chi}_{1})=\chi_1(-1)D_1  %, G(\chi)G(\overline{\chi})=\chi(-1)D
$ in the last equality (see e.g., \cite[Corollary 2.1.47]{Cohenbook}). On the other hand, $U_{D_2}([G_{\ell,\chi_1,\overline{\chi}_{2}}(z),G_{k,\overline{\chi}_{1},\chi_2}(z)]_e)=$
    \begin{align*}
%&\sum_{v~\text{mod}~D_2}D_2^{(k+\ell+2e)/2-1}[G_{\ell,\chi_1,\overline{\chi}_{2}}(z),G_{k,\overline{\chi}_{1},\chi_2}(z)]\bigg|_{k+\ell+2e}\begin{pmatrix}
%            1 & v\\ 0 & D_2
%\end{pmatrix}\\=
&\sum_{v~\text{mod}~D_2}D_2^{(k+\ell+2e)/2-1}\left[G_{\ell,\chi_1,\overline{\chi}_{2}}(z)\bigg|_{\ell}\begin{pmatrix}
            1 & v\\ 0 & D_2
\end{pmatrix}, G_{k,\overline{\chi}_{1},\chi_2}(z)\bigg|_k\begin{pmatrix}1 & v\\ 0 & D_2
\end{pmatrix}\right]_e\\=&\sum_{v~\text{mod}~D_2}D_2^{(k+\ell+2e)/2-1}\left[D_2^{-\ell/2}G_{\ell,\chi_1,\overline{\chi}_{2}}\left(\frac{z+v}{D_2}\right), D_2^{-k/2}G_{k,\overline{\chi}_{1},\chi_2}\left(\frac{z+v}{D_2}\right)\right]_e\\=&\sum_{v~\text{mod}~D_2}D_2^{e}D_2^{-1}\left[G_{\ell,\chi_1,\overline{\chi}_{2}}\left(\frac{z+v}{D_2}\right),G_{k,\overline{\chi}_{1},\chi_2}\left(\frac{z+v}{D_2}\right)\right]_e,
    \end{align*}
    which gives the desired result.
\end{proof}
\begin{proposition}\label{prop:fouriercoexplicit}
    For $n\geq1$, we have 
    \begin{gather*}
        a_{D,\ell,k,e}(n;\chi)=\hspace{-5pt}\sum_{D=D_1D_2}\overline{\chi}_2(-1)D_2^{-e}\hspace{-5pt}\sum_{\substack{a_1,a_2\geq0\\a_1+a_2=nD_2}}\sigma_{\ell-1,\chi_1,\overline{\chi}_2}(a_1)\sigma_{k-1,\overline{\chi}_1,\chi_2}(a_2) \cdot c_{e,a_1,a_2},\\
        c_{e,a_1,a_2}:=\sum_{r=0}^e(-1)^ra_1^ra_2^{e-r}\binom{e+\ell-1}{e-r}\binom{e+k-1}{r}.
    \end{gather*}
\end{proposition}
\begin{proof}
By Proposition \ref{prop:fourierco} and \eqref{eq:defofUmap}, we have
\begin{align*}
    g(n)=\sum_{D=D_1D_2}\overline{\chi}_2(-1)D_2^{-e}h(nD_2),
\end{align*}
where $h(n)$ denotes the $q^{n}$-coefficient of $[G_{\ell,\chi_1,\overline{\chi}_2}(z),G_{k,\overline{\chi}_1,\chi_2}(z)]_e$. Note that
\begin{align*}
    G_{\ell,\chi_1,\overline{\chi}_2}(z)^{(r)}=\sum_{n=0}^{\infty}n^r\sigma_{\ell-1,\chi_1,\overline{\chi}_2}(n)q^n.
\end{align*}
Thus, by \eqref{eq:defofrankin-cohenbracket}, we get $h(nD_2)=$
\begin{align*}
    &\sum_{r=0}^e(-1)^r\binom{e+\ell-1}{e-r}\binom{e+k-1}{r}\hspace{-7.5pt}\sum_{\substack{a_1,a_2\geq0\\a_1+a_2=nD_2}}\hspace{-7.5pt}a_1^r\sigma_{\ell-1,\chi_1,\overline{\chi}_2}(a_1)\cdot a_2^{e-r}\sigma_{\ell-1,\overline{\chi}_1,\chi_2}(a_2)\\=&\sum_{\substack{a_1,a_2\geq0\\a_1+a_2=nD_2}}\sigma_{\ell-1,\chi_1,\overline{\chi}_2}(a_1)\sigma_{\ell-1,\overline{\chi}_1,\chi_2}(a_2) \cdot c_{e,a_1,a_2},   
\end{align*}
giving the desired result.
\end{proof} 
Now, Theorem \ref{thm:mainresultodd} follows immediately from Propositions \ref{prop:k=12l<k}, \ref{prop:k=12l=k} and \ref{prop:fouriercoexplicit}.
\section{Discussion}
First, we would like to mention that one can construct more identities that involve $\tau(n)$ by applying Propositions \ref{prop:nonval<k} and \ref{prop:nonvanl=k} to the cases when $S_{k+\ell+2e}$ has dimension $1$, where $\mathcal{F}_{D,\ell,k,e}$ must be a Hecke eigenform. It is also   possible to obtain more identities by considering Eisenstein series of weight $2$ of higher levels. 

We make some remarks about the construction of $\mathcal{F}_{D,\ell,k,e}$. A natural question is whether we can show that $\mathcal{F}_{D,\ell,k,e}$ is not identically zero without using the Rankin-Selberg method. For example, one can try to show that at least one Fourier coefficient of $\mathcal{F}_{D,\ell,k,e}$ is non-zero. Unfortunately, even the non-vanishing of $a_{p,3,7,1}(1;\chi)$ (see \eqref{eq:1stcop371})  doesn't seem obvious to the author. However, one can approach this question by considering the scenario when the weight $K=k+\ell+2e$ is large. In \cite{twistedNagoya}, the authors considered the case where $e=0$ and $k=\ell$ (the definition of $\mathcal{F}_{D,k,\ell,e}$ needs to be modified), and showed that if $K\geq 10D + 2$, then its first Fourier coefficient is non-zero; see \cite[Proposition 5.6]{twistedNagoya}. It is therefore reasonable to propose the following conjecture.
\begin{conjecture}\label{conj:nonvans=ishingof1stco}
    Let $K=k+\ell+2e$. Then $a_{D,\ell,k,e}(1;\chi)$ is not zero for $K\gg_D1$.
\end{conjecture}

  Finally, note that $\mathcal{F}_{D,\ell,k,e}$ is not a Hecke eigenform in general when $S_{k+\ell+2e}$ has dimension greater than $1$, and the case $\ell<k$  can be proved by 
using the non-vanishing of twisted $L$-values as in Proposition \ref{prop:nonval<k}.
%\begin{proposition}
%    Suppose $\ell<k$. If $\dim S_{k+\ell+2e}>1$ then $\mathcal{F}_{D,\ell,k,e}(z;\chi)$ is not a Hecke eigenform.
%\end{proposition}
The case $\ell=k$ is trickier since we don't know the non-vanishing of twisted central $L$-value $L(f,\chi_D,k+e)$. One special case, $D=1$ and $e=0$ (where we need to modify the definition to the cuspidal projection of Eisenstein series), has been considered in \cite{cuspidal2021}; see \cite[Theorem 4.2]{cuspidal2021}. We propose the following conjecture.
\begin{conjecture}
    Suppose $\dim S_{2k+2e}>1$. Then $\mathcal{F}_{D,k,e}$ is not a Hecke eigenform.
\end{conjecture}
\section*{Acknowledgements}The author thank the anonymous referee for the detailed comments that have improved the exposition of this article.
The author also thanks Hui Xue for the discussion and Erick Ross for kindly providing numerical data.
\bibliographystyle{plain}
  
  \vspace{.2in}

\providecommand{\bysame}{\leavevmode\hbox
to3em{\hrulefill}\thinspace}

\bibliography{ref}

\end{document}